\documentclass{amsart}

\usepackage{amsthm,amsmath,amssymb,amsfonts,epsfig}
\usepackage{hyperref}
\usepackage{stmaryrd}

\theoremstyle{plain}
\newtheorem{lemma}{Lemma}[section]

\newtheorem{theorem}[lemma]{Theorem}

\theoremstyle{definition}
\newtheorem{example}[lemma]{Example}

\newcommand{\N}{{\mathbb N}}

\newcommand{\R}{{\mathbb R}}
\newcommand{\C}{{\mathbb C}}

\newcommand{\cE}{\mathcal{ E}}

\newcommand{\cH}{\mathcal{ H}}

\newcommand{\cL}{\mathcal{L}}

\newcommand{\lam}{\lambda}

\newcommand{\tr}{{\rm tr}}
\newcommand{\Dom}{{\rm Dom}}

\newcommand{\Spec}{{\rm Spec}}
\newcommand{\Span}{{\rm span}}
\newcommand{\Ker}{{\rm Ker}}
\newcommand{\Ran}{{\rm Ran}}

\renewcommand{\Re}{{\rm Re}\;}

\newcommand{\efrac}[2]{\genfrac{}{}{0ex}{}{#1}{#2}}

\newcommand{\conv}{{\rm conv}}
\newcommand{\dist}{{\rm dist}}

\newcommand{\dsc}{\mathrm{dis}}

\newcommand{\ess}{\mathrm{ess}}

\newcommand{\range}{\mathrm{Range}}
\newcommand{\dis}{\mathrm{dis}}

\newcommand{\ud}{\mathrm{d}}

\renewcommand{\div}{\operatorname{div}}
\newcommand{\grad}{\operatorname{grad}}
\newcommand{\curl}{\operatorname{curl}}
\newcommand{\B}{\mathrm{B}}

\newcommand{\slb }{\mathrm{s}}
\newcommand{\cyl}{\mathrm{c}}

\newcommand{\basis}{\frak{B}}

\newcommand{\formT}{\frak{t}}

\newcommand{\exctt}{\eqref{constant_coeff}}
\newcommand{\exvar}{\eqref{non-constant_coeff}}

\newcommand{\wto}{\rightharpoonup}

\begin{document}

\title[Eigenvalue enclosures and convergence for MHD operators]{Eigenvalue enclosures and convergence for the linearized MHD operator}
\author[L. Boulton]{Lyonell Boulton}
\address{Department of Mathematics and Maxwell Institute for Mathematical Sciences, Heriot-Watt University, Edinburgh, EH14 4AS, UK}
\email{L.Boulton@hw.ac.uk}
\author[M. Strauss]{Michael Strauss}
\address{Department of Mathematics and Maxwell Institute for Mathematical Sciences, Heriot-Watt University, Edinburgh, EH14 4AS, UK}
\email{M.Strauss@hw.ac.uk}
\keywords{Eigenvalue enclosures, magnetohydrodynamics, Schur complement, spectral pollution}
\begin{abstract}
We discuss how to compute certified enclosures for the eigenvalues of benchmark linear magnetohydrodynamics (MHD)
operators in the plane slab and cylindrical pinch configurations. For the plane slab, our method relies upon the formulation
of an eigenvalue problem associated to the Schur complement, leading to highly accurate upper bounds for the
eigenvalue. For the cylindrical configuration, a direct application of this formulation is possible, however, it cannot
be rigourously justified. Therefore in this case we rely on a specialized technique based on a method proposed by
Zimmermann and Mertins. In turns this technique is also applicable for finding accurate complementary bounds in the
case of the plane slab. We establish convergence rates for both approaches.
\end{abstract}

\date{\today}
\maketitle


\section{Introduction}

Let $\Omega\subset \R^3$. The linearized ideal MHD equation
\[
    \rho \partial_t^2 \xi(t) + K\xi(t)=0 \qquad \xi(0)=\xi_0 \qquad \partial_t \xi(0) =v_0
\]
for displacement vector $\xi:\Omega\longrightarrow \R^3$  and force operator
\[
      K\xi=\grad [\gamma P(\div \xi)+(\grad P)\cdot \xi]+
      \frac{1}{\mu} [\B\times\curl (\curl(\mathbf{B}\times \xi))-(\curl \mathbf{B})\times \curl(\mathbf{B}\times \xi)]
\]
arises in applications from plasma confinement in thermonuclear fusion. The constants $\mu$ and $\gamma$ here denote the
magnetic permeability and heat ratio. The smooth function $\rho$ is the density, $P>0$ is the pressure and $\mathbf{B}$ the
divergence-free magnetic field of the given equilibrium, satisfying $\mu \grad P=(\curl\mathbf{B})\times\mathbf{B}$.

In the study of this equation a fundamental role is played by the eigenvalue problem associated to $K$.
The appropriate Hilbert space setting ensures that $K$ has a self-adjoint realization. A considerable amount of
research has been devoted to the formulation of a rigorous operator theoretic framework for $K$ and to
the structure of the spectrum, \cite{1989Lifschitz}. Particular attention has been payed to the plane slab (plasma layer) and the cylindrical  (plasma pinch)
configurations \cite{1987Kako,1991Kakoetal,1991Raikov,atla} where $K$ is reduced to a block ordinary
differential operator matrix. A systematic description  (analytical or numerical) of properties of the
eigensolutions turns out to be difficult even for these, the
simplest configurations. This is due to the presence of regions of essential spectrum near the bottom end of
the spectrum.

The plane slab configuration has been the subject of thorough analytical investigation and has become a benchmark
model for the top dominant class of block operator matrix, see \cite{2007Tretter} and references therein.
Precise eigenvalue asymptotics can be found in this case by means of the WKB method
\cite[\S7.5]{1989Lifschitz} or by means of  specialised variational principles, see
\cite[Theorem~3.1.4]{2007Tretter} and references therein. The cylindrical configuration is more
involved due to the presence of singularities in the coefficients of the differential expression;
however, eigenvalue asymptotics are known in this case, \cite{1986Raikov}.

Specialized variational approaches are extremely useful for examining analytic asymptotics for the eigenvalues
in the case of the plasma layer configuration. Unfortunately, as they usually involve a triple variation
formulation, it is arguable whether they are well suited for direct numerical implementations.

If a sequence of subspaces is guaranteed not to produce spectral pollution, then the standard Galerkin method can
be used. A prescribed recipe for avoiding spurious modes when these subspaces are
generated by the finite element method dates back to \cite{1977Rappaz, 1991Kakoetal}. In this classical approach
convergence is guaranteed, however, it is never clear whether a computed eigenvalue is on the left or on the right
of the exact eigenvalue. In this respect the method is not certified.

A technique for finding certified enclosures for the eigenvalues of $K$ in the case of the plane slab configuration
was considered in \cite{2011Strauss} based on the method proposed in \cite{1998Davies, shar, LS2004}. The approach
was based on computing the so-called second order spectrum of $K$ for given finite dimensional subspaces generated
by the spectral basis. In the present paper we consider a further computational strategy which improves upon this
technique in terms of accuracy. Our main approach is to combine two complementary Galerkin-type methods for computing
eigenvalue enclosures which, by construction,  never produce spectral pollution.

For the plane slab, our method relies on the formulation of an eigenvalue problem associated to the Schur complement,
this leads to highly accurate upper bounds. For the cylindrical configuration, a direct application of this formulation
is possible, however, it cannot be rigorously justified. Therefore in this case we rely on a specialized technique
based on a method proposed by Zimmermann and Mertins \cite{zim} as described by Davies and Plum in \cite[Section~6]{dapl}. This approach is intimately related to classical
methods, see \cite{1949Kato,1963Lehmann,1985Goerisch}. 
We also apply this technique to the Schur complement and find accurate complementary lower bounds for the plane slab.

In Section 2 we give a mathematical formulation of the MHD operators under investigation, and some of their spectral
properties. In Section 3 we examine the approximation technique due to Zimmermann and Mertins. We present a formulation of this 
technique in terms of the Galerkin method which establishes both approximation and, importantly, the convergence 
of the method. Our main results are contained in Section 4. We present a highly efficient method for obtaining
upper bounds for eigenvalues above the essential spectrum of top-dominant block operator matrices, an example of which is the matrix $K$ associated to the plasma layer configuration. We show in Theorem \ref{super} that the convergence rate for this approach is the same as 
that achieved by the Galerkin method when applicable (below the essential spectrum). Our method is therefore extremely efficient. 
We also combine this approach with the Zimmermann and Mertins technique to obtain complementary lower bounds for the eigenvalues. 
In Theorem \ref{enclose} we use our results from Section 3 to obtain convergence rates for these lower bounds. In Sections 5 and 6 we apply our 
results to the plasma and cylindrical configurations, respectively.


\section{One-dimensional MHD operators}

The reduction process for the force operator, the precise constraints on the equilibrium quantities and the boundary
conditions on $\Omega$, which yield the one-dimensional boundary value problems associated to the plane slab and
cylindrical configurations, are described in detail in \cite[\S7.2 and \S8.2]{1989Lifschitz}, respectively. Through
this reduction $K$ becomes similar to a superposition of operators which are
self-adjoint extensions of block matrix differential operators of the form
 \begin{equation} \label{matrix}
M_0 = \begin{pmatrix} A & B \\ B^* & D \end{pmatrix}
\end{equation}
acting on $L^2$-spaces of a one-dimensional component.

For the plasma layer, the components of $M_0$ are explicitly given by
 \begin{equation} \label{slab}
\left\{ \begin{aligned}
A&= -\rho_0^{-1}\partial_x \rho_0(v_a^2 + v_s^2)\partial_x + k^2v_a ^2  ,
\\  B&= \begin{pmatrix}(-i\rho_0^{-1}\partial_x\rho_0(v_a^2 + v_s^2)+ig)k_{\perp},
(-i\rho_0^{-1}\partial_x \rho_0v_s^2+ig)k_{\parallel} \end{pmatrix},\\
D&=\begin{pmatrix} k^2v_a^2 + k_{\perp}^2v_s^2
& k_{\perp}k_{\parallel}v_s^2 \\ k_{\perp}k_{\parallel}v_s^2 & k_{\parallel}^2v_s^2
\end{pmatrix},
\end{aligned} \right.
\end{equation}
where
\begin{align*}
\Dom(A) &= H^{2}((0,1);\rho_0 \ud x)\cap H_{0}^{1}((0,1);\rho_0 \ud x)  \\
\Dom(B) &= [H^{1}((0,1);\rho_0 \ud x)]^2  \\
\Dom(D) &= [L^2((0,1);\rho_0\,\ud x)]^2.
\end{align*}
We assume that the Alfv\'en speed $v_a$, the sound speed $v_s$, and the coordinates of the wave vector $k_\perp$ and
$k_\parallel$, are bounded differentiable functions. We also assume that $\rho_0$ and $v_s$ are bounded away from $0$.
Following standard notation in the literature  $k^2(x)=k_\perp^2(x)+k_\parallel^2(x)$  and $g$ is the gravitation constant.
In this case \cite[Proposition~3.1.2]{2007Tretter}
\begin{equation} \label{dom_self}
M_0:[\Dom(A)\cap\Dom(B^*)]\times \Dom(B) \longrightarrow\left[L^2((0,1);\rho_0\,\ud x)\right]^3
\end{equation}
is essentially self-adjoint. Denote by $M_{\slb}$ the closure of $M_0$. Then $M_{\slb}$ is bounded from below, and the essential spectrum is given by
the range of the Alfv\'en frequency $v_a^2k_{\parallel}$ and the mean frequency
$v_a^2v_s^2k_{\parallel}/(v_a^2 + v_s^2)$, see \cite[\S7.6]{1989Lifschitz}.
The discrete spectrum always accumulates at $+\infty$. We show in Example~\ref{constant_coeff} that endpoints of the
essential spectrum can also be points of accumulation.

\begin{example}   \label{constant_coeff}
Let $k_{\|}=\rho=v_a=v_s=g=1$ and $k_\perp=0$. Then $\Spec_\ess(M_\slb)=\{1/2,1\}$ and
$\Spec_\dsc(M_\slb)=\{\lambda_k^\pm\}_{k=1}^\infty$ where
\[
\begin{aligned}
   \lambda_k^+&=1+k^2\pi^2+\sqrt{1+k^2\pi^2+k^4\pi^4} =\mathcal{O}(k^2) & \text{and} \\
   \lambda_k^-&= 1+k^2\pi^2-\sqrt{1+k^2\pi^2+k^4\pi^4}\to\frac12 \qquad &\text{as } k\to\infty.
\end{aligned}
\]
Both $\lambda_k^\pm$ are positive and increasing in $k$. Also
\[
\begin{gathered}
   \operatorname{Range} \int_{(-\infty,3/4]} \!\!\!\!\!\!\!\!\!\!\!\!\!\!\!\!\!\ud E_\lambda =\Span\{ \phi^-_{k}\} , \qquad
   \operatorname{Range} \int_{(3/4,3/2]} \!\!\!\!\!\!\!\!\!\!\!\!\!\!\!\!\! \ud E_\lambda =\Span\{ \phi^1_{k}\}  \\
   \text{and} \qquad \operatorname{Range} \int_{(3/2,\infty)} \!\!\!\!\!\!\!\!\!\!\!\!\!\!\!\ud E_\lambda =\Span\{ \phi^+_{k}\},
\end{gathered}
\]
where
\begin{equation} \label{eigenfunctions}
\begin{gathered}
  \phi^-_{k}(x) = \alpha^-_{k}
   \begin{pmatrix} \sin k\pi x \\ 0 \\ \frac{i}{1-\lambda^-_k}(k\pi \cos k\pi x +\sin k\pi x)    \end{pmatrix}, \quad
    \phi^1_{k}(x) = \frac{1}{\sqrt{2}}
   \begin{pmatrix} 0\\ \sin k\pi x \\ 0   \end{pmatrix} \\
   \text{and} \qquad   \phi^+_{k}(x) = \alpha^+_{k}
   \begin{pmatrix} \sin k\pi x \\ 0 \\ \frac{i}{1-\lambda^+_k}(k\pi \cos k\pi x +\sin k\pi x)    \end{pmatrix},
   \end{gathered}
\end{equation}
the constants $\alpha^\pm_k$ chosen so that $\{\phi_k^\pm\}_{k=1}^\infty$ are normalized. Note that $\{\phi_k^-,\phi_k,\phi_k^+\}$ is an orthonormal basis of $[L^2(0,1)]^3$.
\end{example}

\begin{example}   \label{non-constant_coeff} Let
$\rho_0 = 1$, $k_{\perp} = 1$ , $k_{\parallel} = 1$, $g=1$, $v_a(x)=\sqrt{7/8 - x/2}$, and $v_s(x) = \sqrt{1/8 + x/2}$. The essential spectrum of $M_\slb$ is given by
\[
\Ran(v_a^2k_{\parallel})\cup\Ran(v_a^2v_s^2k_{\perp}/(v_a^2 + v_s^2)) = [7/64,1/4]\cup[3/8,7/8].
\]
Below we will use the fact that $d=\max\Spec(D) = 1+\sqrt{17/32}<\pi^2$.
\end{example}

The plasma pinch configuration yields a differential operator $M_0$ with singular coefficients,
 \begin{equation} \label{cylinder} \left\{
\begin{aligned}
A&= -\partial_r (b^2+\gamma P) \partial_r^\ast + r\left(
\frac{b^2\sin^2 \phi }{r^2}\right)'+b^2k^2_\phi
\\  B&= \begin{pmatrix}(-i\partial_r (b^2+\gamma P) m_\phi +2i b^2k\frac{\sin \phi}{r},
-i\partial_r \gamma P k_\phi \end{pmatrix}\\
D&= \begin{pmatrix} m_\phi^2 (b^2 + \gamma P)+b^2 k^2_\phi
& m_\phi k_\phi \gamma P \\  m_\phi k_\phi \gamma P & k_\phi^2 \gamma P
\end{pmatrix}
\end{aligned} \right.
\end{equation}
acting on $[L^2((0,R_0);r\, \ud r)]^3$. Here $\partial_r^\ast=\frac1r \partial_r r$,
\[
\begin{aligned}
   \mathbf{B}&=(0,b(r) \sin(\phi(r)),b(r)\cos \phi(r)),  \\
   P'(r)&=-b(r)b'(r)=\frac{1}{r} b(r)^2\sin^2 \phi(r),
\end{aligned}
\]
$b(r)$ and $\phi(r)$ are smooth functions with $b'(0)=\phi(0)=0$,
\[
    k_\phi=k \cos \phi + \frac{m}{r} \sin \phi \qquad
\text{and} \qquad m_\phi=\frac{m}{r} \cos \phi -k \sin \phi.
\]
The indices $R_0 k$ and $m$ are integer numbers corresponding to the Fourier
mode decomposition of $K$.

In order to define rigourously the domain of $M_0$ for this configuration, a further
change of variables $r=e^{s}$ is usually implemented, \cite{1987Kako}. Under this change of variables, $M_0$ becomes similar to an operator acting on $[L^2((-\infty, \log R_0);\ud x)]^3$ which is essentially self-adjoint in the space of rapidly decreasing functions at $-\infty$ vanishing at $ \log R_0$, \cite[Theorem~2.3]{1987Kako}.  Operator $M_0$ is essentially self-adjoint in the pre-image of this space under the similarity transformation. We denote by $M_\cyl$ the unique self-adjoint extension of $M_0$ in the latter domain.

The original formulation \eqref{cylinder} is numerically more stable for the treatment of the eigenvalues via a projection method. The finite element space generated by Hermite elements of order $3$, $4$ and $5$, subject to Dirichlet boundary conditions  at $0$ and $R_0$, considered below are $C^1$-conforming and hence are all contained in $\Dom(M_{\cyl})$ for the benchmark equilibrium quantities considered in our examples.

The essential spectrum of $M_{\cyl}$ consists of an Alfv\'en band determined by the range of $b^2 k_\phi^2$, and a slow magnetosonic band determined by the range of
$(b^2 k_\phi^2 \gamma P)/(b^2+\gamma P)$, \cite[Theorem~3.5]{1987Kako}. As in the previous configuration, these bands are located near the bottom of the spectrum and $+\infty$ is always an accumulation point if the discrete spectrum.

\begin{example} \label{example_cylinder}
Let $P\equiv 0$, $b\equiv 1$, $\phi\equiv 0$, and $k=m=1$. Then $\Spec_\ess(M)=\{0,1\}$
(the point $0$ is the slow magnetosonic spectrum and the point $1$
is the Alfv\'en spectrum). On the other hand $\Spec_\dsc (M)=\{E^2+1:J'_1(E)=0\}$
where $J_v(x)$ is the Bessel function of index $v$.
\end{example}


\section{Pollution-free bounds for eigenvalues}\label{zimmerman}

The essential spectrum of both operators $M_{\slb}$ and $M_{\cyl}$ is non-negative. Therefore unstable
spectrum can only occur in the discrete spectrum. The eigenvalues below the bottom of the essential
spectrum can be computed using the standard Galerkin method. Hence, the stability of the configuration can
be determined by means of the Rayleigh-Ritz variational principle.

By contrast, computing the eigenvalues above the essential spectrum is problematic due to
the possibility of variational collapse. The technique described in this section avoids spectral pollution 
and can be implemented on the finite element method. It gives certified enclosures up to machine precision for eigenvalues above
the essential spectrum. In Section~\ref{supersection} we argue that this technique should be applied,
not only to $M_{\slb}$, but also to its Schur complement. In order to keep a neat notation, we formulate
the general procedure for a generic semi-bounded self-adjoint operator $T$ acting on a Hilbert space
$\mathcal{H}$.


\subsection{Basic notation}

Let the dense subspace  $\Dom(T)\subset \mathcal{H}$ be the domain of $T$. For $\mu=\min\Spec(T)$, let $\formT$ be the close bilinear form induced by the non-negative operator $T-\mu$ with domain $\Dom(\formT)=\Dom(| T|^{\frac{1}{2}})$. We denote the inner products and norms that render $\Dom(T)$ and  $\Dom(\formT)$ with a Hilbert space structure, respectively by $\langle u,v\rangle_{T} = \langle Tu,Tv\rangle + \langle u,v\rangle$,  $\langle u,v\rangle_{\formT} = \formT(u,v)+\langle u,v\rangle$, $\|\cdot \|_{T}$ and $\|\cdot\|_{\formT}$.

Let $\cE$ be a subspace of $\Dom(T)$. For another subspace $\mathcal{L}$, we denote
\begin{align*}
\delta(\cE,\cL) &= \sup_{\phi\in\cE,~\|\phi\|=1}\dist[\phi,\cL]
& \text{if } \cL\subset \mathcal{H} \\
\delta_{\formT}(\cE,\cL) &= \sup_{\phi\in\cE,~\|\phi\|_{\formT}=1}\dist_{\formT}[\phi,\cL] & \text{if } \cL\subset \Dom(\formT)\\
\delta_T(\cE,\cL) &= \sup_{\phi\in\cE,~\|\phi\|_T=1}\dist_T[\phi,\cL]
& \text{if } \cL\subset \Dom(T).
\end{align*}
Here and elsewhere $\dist_\bullet [\phi,\cL]$ refers to the Haussdorff distance in the norm $\|\cdot\|_\bullet$ between $\{\phi\}$ and $\cL$.

Below we establish spectral approximation results by following the classical framework of \cite{chat}. These results will be formulated in a general context for sequences of subspaces $\cL_n \subset \cH$ which are dense as $n\to \infty$ in the following precise senses.
We will say
\begin{align*}
   (\cL_n)\in \Lambda \equiv \Lambda(I) & \quad \iff \quad \dist[u,\cL_n]\to 0 & \forall u\in\cH \\
   (\cL_n)\in\Lambda(\formT) & \quad \iff \quad \dist_{\formT}[u,\cL_n]\to 0 & \forall u\in \Dom(\formT)\\
   (\cL_n)\in\Lambda(T) & \quad \iff \quad \dist_{T}[u,\cL_n]\to 0 & \forall u\in \Dom(T).
\end{align*}

Let $\cL\subset \Dom(\formT)$. Below we denote by $\Spec(T,\cL)$ the spectrum of the classical weak Galerkin problem:
\begin{equation*} \label{galerkin}
\exists~u\in\cL\backslash\{0\} \text{ and }\mu\in\mathbb{R}\quad\text{such that}\quad \formT(u,v) = \mu\langle u,v\rangle\quad\textrm{for all}\quad v\in\cL.
\end{equation*}

Assume that an interval $(a,b)\subset \R$ is such that $\operatorname{Tr}[\int_{(-\infty,a)} \ud E_\lambda]=\operatorname{Tr}[\int_{(b,\infty)} \ud E_\lambda]=\infty$.
Then for general $(\cL_n)\in \Lambda(T)$, the set
\[
   (a,b) \cap \bigcap_{n=1}^{\infty} \overline{\bigcup _{k\geq n} \Spec(T,\cL_k)}
\]
could be a much larger set than $(a,b)\cap \Spec(T)$.
This phenomenon is usually called spectral pollution. See \cite{1977Rappaz} for further details on this in case $T=M_{\slb}$ and $\cL_n$ chosen as finite element spaces.

The following classical convergence result will play a fundamental role below, see \cite[Theorem 6.11]{chat}. Assume that $\|T\|<\infty$. Let $(\cL_n)\in\Lambda$. For any isolated eigenvalue $\lambda\in \Spec(T) \setminus \conv[\Spec_\ess(T)]$,
\begin{equation}\label{galerkin2}
\delta(\ker(T-\lambda),\cL_n)=\varepsilon_n\to 0 \quad \Rightarrow \quad\dist[\lambda,\Spec(T,\cL_n)] = \mathcal{O}(\varepsilon_n^2).
\end{equation}

If $T$ is only bounded from below, the condition $\cL_n\subset\Dom(\formT)$ and $(\cL_n)\in\Lambda$, is typically not sufficient to ensure approximation. 
By applying the spectral mapping theorem  it can be shown that \eqref{galerkin2} still holds true for a
$\lambda<\min \Spec_{\mathrm{ess}}(T)$ whenever $(\cL_n)\in\Lambda(\formT)$
and $\delta$ is replaced by $\delta_\formT$, see for example the trick applied in the proof of \cite[Corollary~3.6]{2010bs}.
This type of convergence is often called superconvergence.


\subsection{The Zimmermann-Mertins method}\label{zimmermethod}
The following method for computing eigenvalue enclosures originated from \cite{zim} and is closely related to the
classical Lehmann method.  Below we show that it may be described in simple terms by means of mapping theorems at the
level of reducing spaces for the resolvent. As we will see subsequently, convergence estimates will follow easily from
\eqref{galerkin2}.

This method turns out to be efficient for computing eigenvalue enclosures for $M_{\slb}$ and $M_{\cyl}$ in their original
matrix formulation. In Section~\ref{supersection} we will discuss a further technique which allows improvement in accuracy
for $M_{\slb}$ and depends on re-writing the eigenvalue problem in terms of the Schur complement.

According to \cite[Theorem~11]{dapl} the present approach is equivalent with optimal constant to another method formulated in
\cite{1998Davies}. The latter is closely related to the so-called second order relative spectrum \cite{shar} which was applied to $M_\slb$ in \cite{2011Strauss}.
We should stress that the latter is probably best for obtaining preliminary information about
the spectrum, \cite{2010bs, 2011bStrauss}. This \emph{a priori} information includes a reliable
guess on the interval $[a,b]$ below.

Let $a<b$ be such that $(a,b)\cap \Spec(T)\not= \varnothing$. Assume that a finite-dimensional subspace $\cL \subset \Dom(T)$ is such that
 \begin{equation}\label{zimcon}
    \min_{u\in \cL}\frac{\langle Tu,u \rangle}{\|u\|^2}<b \quad\textrm{and}\quad
    \max_{u\in \cL}\frac{\langle Tu,u \rangle}{\|u\|^2}>a.
\end{equation}
Note that this condition is certainly satisfied by $\cL=\cL_n$ for $n$ sufficiently large, if $(\cL_n)\in \Lambda(\formT)$. We define two inverse residuals associated to the interval $(a,b)$:
\begin{equation}\label{zimmer}
\tau^+=\max_{u\in \cL} \frac{\langle  (T-a)u,u  \rangle}{\langle  (T-a) u,(T-a) u  \rangle}\quad
\text{and} \quad\tau^-=\min_{u\in \cL} \frac{\langle  (T-b)u,u  \rangle}{\langle  (T-b) u,(T-b) u  \rangle}.
\end{equation}
By virtue of \eqref{zimcon}, we have $\tau_+>0$ and $\tau_-<0$.

\begin{lemma} \label{non-pollu}
Suppose that $[a,b]\cap \Spec(T)= \{\lambda\}$ and $a<\lambda<b$. For any subspace $\cL\subset \Dom(T)$ satisfying \eqref{zimcon} the inverse residuals \eqref{zimmer} are such that
\begin{equation}\label{zimbound}
b + \frac{1}{\tau^-}\le\lambda\le a+\frac{1}{\tau^+}.
\end{equation}
\end{lemma}
\begin{proof}
We prove the first inequality, the second my be proved similarly. Let
$\hat{\cL} = (T-b)\cL$, then
\begin{align*}
(\lambda - b)^{-1}&= \min[\Spec(T-b)^{-1}] = \min_{v\in \Dom(T)}\frac{\langle(T-b)^{-1}v,v \rangle}{\|v\|^2}  \\  & \leq  \min_{v\in\hat{\cL}}\frac{\langle(T-b)^{-1}v,v \rangle}{\|v\|^2}
= \min_{u\in \cL} \frac{\langle  (T-b)u,u  \rangle}{\langle  (T-b) u,(T-b) u  \rangle}=\tau^-.
\end{align*}
\end{proof}

Note that \[\tau^-=\min[\Spec((T-b)^{-1},\hat\cL)]\quad\textrm{and}\quad\tau^+=\max[\Spec((T-a)^{-1},(T-a)\cL)].\]
This observation turns out to be useful when studying convergence of the enclosure \eqref{zimbound}
as $\cL$ increases in dimension. Below $\tau^{\pm}_n=\tau^{\pm}$ for $\cL_n=\cL$.

\begin{lemma}\label{daplcon}
Let $\lambda\in\Spec_{\dis}(T)$ and
assume that $[a,b]\cap\Spec(T) = \{\lambda\}$ with $a<\lambda<b$. Let $\basis$ be an orthonormal basis of
$\ker(T-\lambda)$, and $(\cL_n)\in\Lambda(T)$ be such that $\dist_T[\basis,\cL_n] = \varepsilon_n\to 0$. For all sufficiently large $n\in\mathbb{N}$
\begin{equation}\label{zimconv}
b + \frac{1}{\tau_n^-}\le\lambda\le a+\frac{1}{\tau_n^+}\quad\text{and}\quad\left(a+\frac{1}{\tau_n^+}\right) - \left(b + \frac{1}{\tau_n^-}\right) = \mathcal{O}(\varepsilon_n^2).
\end{equation}
\end{lemma}
\begin{proof}
The condition \eqref{zimcon} is satisfied for all sufficiently large $n\in\mathbb{N}$, so the left hand side of
\eqref{zimconv} follows from Lemma~\ref{non-pollu}.

Let $\hat{\cL}_n=(T-b)\cL_n$. Since $(\cL_n)\in\Lambda(T)$, it follows that $(\hat{\cL}_n)\in \Lambda$.
Let $\basis = \{\phi_1,\dots,\phi_{k}\}$. Then there exist vectors $u_{n,j}\in\cL_n$, such that
$\|(T-b)(\phi_j - u_{n,j})\|\le (1+| b|)\varepsilon_n$
for each $1\le j\le k$. Set $\hat{u}_{n,j} = (T-b)u_{n,j}\in\hat{\cL}_n$, then for any normalised
$\phi\in\ker(T-\lambda)$ we have
\begin{align*}
\Big\|\phi - \sum_{j=1}^k\frac{\langle\phi_j,\phi\rangle}{\lambda-b}\hat{u}_{n,j}\Big\| &=
\Big\|\sum_{j=1}^k\langle\phi_j,\phi\rangle\phi_j -
\sum_{j=1}^k\frac{\langle\phi_j,\phi\rangle}{\lambda-b}\hat{u}_{n,j}\Big\|\\
&=\Big\|(T-b)\sum_{j=1}^k\frac{\langle\phi_j,\phi\rangle}{\lambda-b}(\phi_j - u_{n,j})\Big\|\\
&\le \frac{k(1+| b|)\varepsilon_n}{b-\lambda}.
\end{align*}
Thus $\delta(\ker(T-\lambda),\hat{\cL}_n)\le k(1+| b|)\varepsilon_n/(b-\lambda)$. Hence, applying \eqref{galerkin2}
to the operator $(T-b)^{-1}$ and eigenvalue $(\lambda - b)^{-1}$ yields
$|\tau_n^- - (\lambda - b)^{-1}| = \mathcal{O}(\varepsilon_n^2)$, and therefore
$b+\frac{1}{\tau_n^-}-\lambda=\mathcal{O}(\varepsilon_n^2)$. Similarly we have
$a+\frac{1}{\tau_n^+}-\lambda=\mathcal{O}(\varepsilon_n^2)$, and the right hand side of \eqref{zimconv} follows.
\end{proof}

By means of an example we now show that $(\cL_n)\in\Lambda(\formT)$ and $(\cL_n)\subset\Dom(T)$, is not generally sufficient to
ensure a decrease in the size of the enclosure as $n\to \infty$. The crucial point here is that
$(\cL_n)\in\Lambda(\formT)$ does not ensure that $\dist_T(\basis,\cL_n)$ decrease to $0$.

\begin{example}\label{ex3}
Let $T=M_\slb$ be as in Example~\ref{constant_coeff}. For $j\in \N$, let $\phi_{3k-2}=\phi^-_k$, $\phi_{3k-1}=\phi^1_k$ and $\phi_{3k}=\phi^+_k$, where the right hand sides are given by \eqref{eigenfunctions}. Let $\lambda_{3k-2}=\lambda^-_k$, $\lambda_{3k-1}=1$ and $\lambda_{3k}=\lambda_k^+$. Then $T\phi_j=\lambda_j \phi_j$ and
$\{\phi_j\}$ is an orthonormal basis of $\mathcal{H}=[L^2(0,1)]^3$.

For $n>2$ consider the subspaces
$\cL_n = \Span\{\alpha_n \phi_1+\varepsilon_n\phi_{3n},\phi_2,\dots,\phi_{3n-1}\}$ where $\varepsilon_n=\frac{1}{\lambda_n^+}$ and $\alpha_n=\sqrt{1-\varepsilon_n^{2}}$.
Then $\cL_n\subset\Dom(T)$ for every $n\in\mathbb{N}$. We show that $(\cL_n)\in\Lambda(\formT)$. Let $u\in\Dom(\formT)$ and $\gamma_j=\langle u,\phi_j\rangle$. Then
\[
\quad\| u\|_{\formT} = \sqrt{\sum_{j=1}^\infty (1+\lambda_j )|\gamma_j|^2}<\infty.
\]
Let $u_n=\gamma_1(\alpha_n\phi_1+\varepsilon_n\phi_{3n}) + \sum_{j=2}^{3n-1}\gamma_j\phi_j$. Then
\begin{align*}
\| u - u_n\|_{\formT}^2 &= \Big\|\gamma_1(\alpha_n-1) \phi_1 +
(\gamma_1\varepsilon_n-\gamma_{3n})\phi_{3n} - \sum_{j=3n+1}^\infty\gamma_j\phi_j\Big\|_{\formT}^2\\
&=(1+\lambda_1^-)|\gamma_1|^2|\alpha_n-1|^2 + (1+\lambda_n^+)|\gamma_1\varepsilon_n-\gamma_{3n}|^2 + \sum_{j=3n+1}^\infty (1+\lambda_j)|\gamma_j|^2\\
&\le(1+\lambda_1^-)|\gamma_1|^2|\alpha_n-1|^2 + \frac{(1+\lambda_n^+)|\gamma_1|^2}{(\lambda_n^+)^2}  + \sum_{j=3n}^\infty (1+\lambda_j)|\gamma_j|^2\longrightarrow 0,
\end{align*}
therefore $(\cL_n)\in\Lambda(\formT)$.

This ensures that $\dist_{\formT}(\phi_1,\cL_n)\to 0$. On the other hand, a straightforward calculation shows that
$\dist_T(\phi_1,\cL_n)\to 1+\sqrt{3}$ as $n\to \infty$. Choose $a=0$ and $b=\frac{\lambda_1^-+\lambda_2^-}{2}$ so that $[a,b]\cap\Spec(T)=\{\lambda_1^-\}$.
For $n$ large enough,
\[
\begin{gathered}
\frac{\langle T(\alpha_n\phi_1+\varepsilon_n\phi_{3n}),\alpha_n\phi_1+\varepsilon_n\phi_{3n}\rangle}
{\|\alpha_n\phi_1+\varepsilon_n\phi_{3n}\|^2}  - b=\alpha_n^2 \lambda_1^-+\frac{1}{\lambda_n^+}-
\left(  \frac{\lambda_1^-+\lambda_2^-}{2} \right) < 0\quad\textrm{and} \\
\frac{\langle T\phi_{n-1},\phi_{n-1}\rangle}
{\|\phi_{n-1}\|^2}  - a=\lambda_n^+ > 0.
\end{gathered}
\]
Thus \eqref{zimcon} is satisfied and we may obtain upper and lower bounds on the eigenvalue $\lambda_1^-$ from the left side of \eqref{zimconv}.

Let us now prove that the length of the enclosure in \eqref{zimconv} does not decrease. Indeed
\[
\begin{gathered}
\tau_n^+=\frac{\alpha_n^2 \lambda_1^-+ (\lambda_n^+)^{-1}}{\alpha_n^2(\lambda_1^-)^2+1}\to
\frac{\lambda_1^-}{(\lambda_1^-)^2+1}<(\lambda_1^-)^{-1} \\
\tau_n^-=\frac{\alpha_n^2(\lambda_1^- -b)+\varepsilon_n^2(\lambda_n^+-b)}{\alpha_n^2(\lambda_1^- -b)^2+\varepsilon_n^2(\lambda_n^+-b)^2} \to \frac{(\lambda_1^- -b)}{(\lambda_1^- -b)^2+1}<
(\lambda_1^- -b)^{-1}
\end{gathered}
\]
as $n\to \infty$.
\end{example}

It is easily verified that
\[
\delta_T(\ker(T-\lambda),\cL_n)=\mathcal{O}(\varepsilon_n)\qquad \Rightarrow
 \qquad\delta_{\formT}(\ker(T-\lambda),\cL_n)=\mathcal{O}(\varepsilon_n).
\]
As the following example shows, $\delta_T(\ker(T-\lambda),\cL_n)$ and $\delta_{\formT}(\ker(T-\lambda),\cL_n)$ can converge
at the same rate, but the latter may be faster. Therefore, there is a potential loss in convergence of the method when compared
with the standard Galerkin method in the case where the latter is applicable. This loss of convergence is compensated by the
fact that the enclosures found are certified and free from spectral pollution.

\begin{example}\label{ex1}
Let $T$ and $\phi_n$ be as in Example~\ref{ex3}. Let
$\varepsilon_n=(\lambda_n^+)^{-2}$ and $\alpha_n = \sqrt{1-\varepsilon_n^{2}}$.
If we consider
$\cL_n = \Span\{\alpha_n\phi_1+\varepsilon_n\phi_{3n-1},\phi_2,\dots,\phi_{3n-2}\}$, then $(\cL_n)\in\Lambda(T)$, and both
$\delta_{\formT}(\ker(T-1),\cL_n)$ and
 $\delta_{T}(\ker(T-1),\cL_n)$ are $\mathcal{O}(n^{-2})$. If we consider
 $\cL_n = \Span\{\alpha_n\phi_1+\varepsilon_n\phi_{3n},\phi_2,\dots,\phi_{3n-1}\}$,
 then $(\cL_n)\in\Lambda(T)$ once again but now
$\delta_{\formT}(\ker(T-\lambda_1^-),\cL_n)= \mathcal{O}(n^{-3})$ and
 $\delta_{T}(\ker(T-\lambda_1^-),\cL_n)=\mathcal{O}(n^{-2})$.
\end{example}


\section{Operator matrices and eigenvalue enclosures} \label{supersection}

The linearized MHD operator associated to the plasma layer configuration \eqref{slab} falls into the class of top dominant
block matrices. We show that enclosures for the eigenvalues of $M_\slb$ which lie above the essential spectrum can be obtained from
enclosures for the eigenvalues of its Schur complement.  Denoted  by $S(\mu)$, the latter is a $\mu$-dependant holomorphic family of
semi-bounded operators. Upper bounds for its eigenvalues can be found from a direct application of
the Galerkin method. We show in Section~\ref{evub} that these upper bounds are
superconvergent as the dimension of $\cL_n$ increases, hence they turn out to be asymptotically sharper than the upper
bounds found from the method of Section~\ref{zimmerman} applied directly to $M_\slb$. In Section~\ref{evlb}, on the
other hand, we show how to find lower bounds for the eigenvalues of $M_\slb$ from corresponding lower bounds on the
eigenvalues of $S(\mu)$. The latter are found from the left side of \eqref{zimbound} with $T=S(\mu)$ for a particular choice of $\mu$.

\subsection{Basic notation}
The results established below apply to any block operator matrix $M_0$ as in \eqref{matrix} which is top dominant in the following precise sense, see \cite{2007Tretter}.
\begin{enumerate}
\item \label{top_1} $A$ and $D$ are self-adjoint operators acting on Hilbert spaces $\cH_1$ and $\cH_2$, respectively.
\item  \label{top_2} $A$ is bounded from below, $D$ is bounded from above, $B$ is closed and densely defined on a domain of $\cH_2$
                      with values in $\cH_1$.
\item  \label{top_3} $\Dom(| A|^{\frac{1}{2}})\subset\Dom(B^*)$, $\Dom(B)\subset\Dom(D)$ and $\Dom(B)$ is a core for $D$.
\end{enumerate}
Without further mention, we assume that the entries of $M_0$ are subject to these conditions.
They are satisfied by the plane slab configuration MHD operator, however, for $m\not=0$  the ansatz \ref{top_3}
does not hold in general for the cylindrical pinch configuration.

The first condition in \ref{top_3} and the semi-boundedness of $A$, together, imply the existence of
constants $\alpha,\beta\ge 0$ such that
\begin{equation}\label{domcons}
\| B^*u\|^2 \le \alpha\frak{a}[u] + \beta\| u\|^2\quad\textrm{for all}\quad u\in\Dom(| A|^{\frac{1}{2}})=\Dom(\frak{a})
\end{equation}
where $\frak{a}$ is the closure of the quadratic form associated
to $A$.  These two constraints also imply that $(A-\nu)^{-1}B$ is a bounded operator, so $\Dom(\overline{(A-\nu)^{-1}B})=\mathcal{H}_2$, for an arbitrary $\nu<\min\Spec(A)$. The self-adjoint closure of $M_0$, which we denote here
by $M$, is explicitly given by
\begin{align}
\Dom(M)&=\bigg\{\left(
\begin{array}{c}
x\\
y
\end{array} \right):y\in\Dom(D),~x + \overline{(A-\nu )^{-1}B}y\in\Dom(A)\bigg\}\label{top1}\\
M\left(
\begin{array}{c}
x\\
y
\end{array} \right) &= \left(
\begin{array}{c}
A(x + \overline{(A-\nu )^{-1}B}y) - \nu\overline{(A-\nu )^{-1}B}y\\
B^*x + Dy
\end{array} \right)\label{top2},
\end{align}
see \cite[Section 4.2]{math0} and references therein.

Set $d=\max\Spec(D)$ and $U=\{z\in\mathbb{C}:\Re z >d\}$. For $\mu\in U$ consider the following  family of forms
\begin{equation}\label{schurform}
\frak{s}(\mu)[x,y] = \frak{a}[x,y] - \mu\langle x,y\rangle - \langle(D-\mu)^{-1}B^*x,B^*y\rangle
\end{equation}
with common domain $\Dom(\frak{s})=\Dom(\frak{s}(\mu))=\Dom(\frak{a})$. Then $\frak{s}(\mu)$ is a holomorphic family of type (a),
see \cite[Proposition 2.2]{math}. Associated to these forms is a
holomorphic family of type (B) sectorial operators $S(\mu)$:
\begin{align}
\Dom(S(\mu)) &= \{x\in\Dom(\frak{a}): x - \overline{(A-\nu)^{-1}B}(D-\mu)^{-1}B^*x\in\Dom(A)\}\label{schurdom}\\
S(\mu)x &= (A-\nu)(x - \overline{(A-\nu)^{-1}B}(D-\mu)^{-1}B^*x) + (\nu-\mu)x\label{schuract};
\end{align}
see \cite[Proposition 4.4]{math0}. 
Here, as above, $\nu<\min\Spec(A)$ is fixed, but can be chosen arbitrarilly.
We note that for any $x\in\Dom(\frak{s})$ and $\mu\in U\cap\mathbb{R}$, 
\begin{equation}\label{derivative}
\frac{\partial \frak{s}(\mu)[x]}{\partial \mu} = -\| x\|^2 - \|(D-\mu)^{-1}B^*x\|^2.
\end{equation}
The families $\frak{s}(\cdot)$ and $S(\cdot)$ are called the Schur form and the Schur
complement associated to $M_0$, respectively.

The form $\frak{s}(\mu)$ is symmetric and semi-bounded whenever $\mu\in \R\cap U$. The corresponding operator
$S(\mu)$ is therefore self-adjoint and bounded from below. We set the spectra of the Schur complement as
\begin{align*}
\Spec(S)&=\{\mu\in U:0\in\Spec(S(\mu))\}, \\
\Spec_\dsc(S)&=\{\mu\in U:0\in\Spec_\dsc(S(\mu))\} \qquad and \\
\Spec_\ess(S)&=\{\mu\in U:0\in\Spec_\ess(S(\mu))\}.
\end{align*}


\subsection{Upper bounds via Schur complement}  \label{evub}
We denote
\[
\lambda_{\mathrm{e}}=\inf\{\Spec_{\ess}(M)\cap(d,\infty)\}
\] and $\lambda_1\le\lambda_2\le\cdots$ the
repeated eigenvalues of $M$ which lie in the interval $(d,\lambda_{\mathrm{e}})$.

\begin{lemma}
The spectra of $S$ and $M$ coincide on $(d,\lambda_{\mathrm{e}})$ and $\Spec_{\ess}(S)\cap(d,\lambda_{\mathrm{e}}) = \varnothing$. Moreover
$\dim\Ker(S(\lambda_j)) = \dim\Ker(M-\lambda_j)$.
\end{lemma}
\begin{proof}
For the first and third assertions, see \cite[Proposition 4.4]{math0}. 
For the second assertion we proceed by contradiction. 

Suppose there exists $\lambda\in (d,\lambda_{\mathrm{e}})$ such that $0\in \Spec_\ess(S(\lambda))$.
Since $S(\lambda)=S(\lambda)^*$, there is a singular Weyl sequence $x_n\in \Dom(S(\lambda))$ such that $\|x_n\|=1$, $x_n\wto 0$ and $\|S(\lambda)x_n\|\to 0$. Let
\begin{align*}
y_n=\begin{pmatrix}
x_n\\
-(D-\lambda_j)^{-1}B^*x_n
\end{pmatrix}.
\end{align*}
Then $y_n\in\Dom(M)$ and $\| y_n\|\ge 1$.  A direct calculation shows that
\[
   (M-\lambda) y_{n}=\begin{pmatrix}  S(\lambda)x_n \\ 0 \end{pmatrix} \to 0.
\]

We prove that $y_n$ has a subsequence $y_{n(k)}\wto 0$, which in turn is a contradiction because $\lambda\not \in\Spec_\ess(M)$. Let $\mathcal{D}=(D-\lambda)\Dom(B)$. By virtue of the second and third ansatz in \ref{top_3}, $\mathcal{D}$ is a dense subspace of $\mathcal{H}_2$. Moreover,
\begin{equation}  \label{weak_dense}
       \langle (D-\lambda)^{-1}B^* x_n,y \rangle \to 0 \qquad \text{for all } y\in \mathcal{D}.
\end{equation}
According to \eqref{domcons},
\begin{align*}
   \|B^* x_n \|^2 & \leq  \alpha \frak{a}[x_n]+\beta \leq \alpha( \frak{a}[x_n] -\lambda) +\alpha\lambda +\beta \\
   &\leq \alpha \frak{s}(\lambda)[x_n] +\alpha\lambda +\beta = \alpha\langle S(\lambda)x_n,x_n\rangle +\alpha\lambda +\beta.
\end{align*}
As the right hand side of this identity is uniformly bounded for all $n$, there exists a subsequence $x_{n(k)}$
and $z\in \mathcal{H}_2$ such that $(D-\lambda)^{-1} B^* x_{n(k)}\wto z$. Since \eqref{weak_dense} implies that $z=0$, the subsequence $y_{n(k)}$ is as needed. 
\end{proof}

For $\mu>d$, we denote the spectral subspace of $S(\mu)$ corresponding to an interval $J$ by
\[
  \cE_{J}(S(\mu))=\range \int_{J} \ud E_\lambda.
\]
Here we abuse the notation and write $\ud E_\lambda$ for the spectral measure associated to the self-adjoint
operator $S(\mu)$ also. Let the dimension of $\cE_{(-\infty,0)}(S(\mu))$ be
\[
    \kappa(\mu) =\tr  \int_{(-\infty,0)} \ud E_\lambda.
\]
Throughout this section we assume that $\kappa(\mu)<\infty$ for some $\mu>d$. By \cite[Theorem 4.5]{math0}, this
assumption and \ref{top_1}-\ref{top_3} imply the existence of $\gamma>d$ such that $(d,\gamma]\cap \Spec(M)=\varnothing$. We write
$\kappa:=\kappa(\gamma)<\infty$ and note that $\kappa$ is independent of the particular choice of
$\gamma\in (d,\min\{\Spec(M)\cap(d,\infty)\})$.

Let $l_1(m) = \min\{j\in\mathbb{N}:\lambda_j = \lambda_m\}$
and $l_2(m) = \max\{j\in\mathbb{N}:\lambda_j = \lambda_m\}$. Then $\kappa(\cdot)$ is constant on intervals contained
in $(d,\lambda_{\mathrm{e}})\setminus \Spec(M)$ and
\begin{align}
&\dim\cE_{(-\infty,0)}(S(\lambda_m)) = \kappa + l_1(m) - 1,\label{dims1}\\
&\dim\cE_{(-\infty,0]}(S(\lambda_m)) = \kappa + l_2(m)\label{dims0};
\end{align}
see \cite[Section 2]{math0} for further details.

\begin{example}   \label{ex_case_plasma}
In the case of the plasma layer configuration, $S(\mu)$ is a family of
Sturm-Liouville operators. It is readily seen from the results of \cite[\S7.5]{1989Lifschitz} that $\lambda_{\mathrm{e}}=\infty$,  $\kappa<\infty$ and
$\Spec(M)\cap(d,\lambda_{\mathrm{e}})$ consists of a sequence of simple eigenvalues which accumulate at $+\infty$.
\end{example}

We now describe the theoretical framework and basic procedure for approximating
a fixed eigenvalue $\lambda_m$. Denote by
$E_1(\mu)\le\dots\le E_{\kappa+m}(\mu)$ the first $\kappa+m$ eigenvalues of
$S(\mu)$ repeated according to their multiplicity. Let
\[
\cL=\Span\{u_1,\dots,u_n\}\subset \Dom(\frak{s}) \quad\textrm{where}\quad \langle u_i, u_j\rangle =\delta_{ij},
\]
be an $n$-dimensional subspace where $n\ge \kappa+m$. Consider the family of matrices $S_{\cL}(\mu)\in \C^{n\times n}$ whose entries are given by
\[
S_{\cL}(\mu)_{i,j} = \frak{s}(\mu)[u_j,u_i] \qquad i,j=1,\ldots,n.
\]
Denote by $E_1(\cL,\mu)\le\dots\le E_{\kappa+m}(\cL,\mu)$ the first $\kappa+m$ eigenvalues of
$S_{\cL}(\mu)$ repeated according to their multiplicity.

\begin{lemma}\label{upbdslem}
Let $\mu\in(d,\infty)$ be such that $E_{\kappa+m}(\cL,\mu) \le 0$,
then $\lambda_m\le\mu$.
\end{lemma}
\begin{proof}
We suppose that $\lambda_m>\mu$. The Rayleigh-Ritz variational principle ensures that
$E_{\kappa+m}(\mu) \leq  E_{\kappa + m}(\cL,\mu) \le 0$, and therefore
\begin{equation}\label{dims2}
\dim\cE_{(-\infty,0]}(S(\mu))\ge\kappa + m.
\end{equation}
If $\mu=\lambda_j$ for some some $1\le j<l_1(m)$, then from \eqref{dims0} we have $\dim\cE_{(-\infty,0]}(S(\mu))
 = \kappa + l_2(j) < \kappa + l_1(m)$, which contradicts \eqref{dims2}. Suppose now that $\mu\not \in\Spec(M)$.
Then from \eqref{derivative} we have $\frak{s}(\mu)\ge\frak{s}(\lambda_m)$
from which it follows that $\dim\cE_{(-\infty,0)}(S(\mu))\le\dim\cE_{(-\infty,0)}(S(\lambda_m))$. From \eqref{dims1}
we then deduce that $\dim\cE_{(-\infty,0)}(S(\mu))<\kappa + m$, which contradicts \eqref{dims2}.
\end{proof}

An upper bound for $\lambda_m$ may be obtained by applying the Galerkin method to the Schur
complement, then finding a $\mu\in(d,\infty)$ such that $S(\mu)$ has at least $\kappa+m$ non-positive eigenvalues
via a root finding algorithm. We now turn our attention to the convergence properties of this approach. For this we employ \eqref{galerkin2} assuming $T=S(\lambda_m)$ and denote $\cE=\ker(T)$.

\begin{theorem}\label{super}
Let $(\cL_n)\in\Lambda(\formT)$ be such that
$\delta_{\formT}(\cE,\cL_n)=\varepsilon_n\to 0$ as $n\to\infty$.
Let $\mu_n^+\in\mathbb{R}$ be such that $E_{\kappa+m}(\cL_n,\mu_n^+) = 0$. Then $\mu_n^+>\lambda_m$ and $\mu_n^+-\lambda_m = \mathcal{O}(\varepsilon_n^2)$.
\end{theorem}
\begin{proof}
From \eqref{dims1} it follows that $S(\lambda_m)$ has $\kappa + l_1(m)-1$ negative eigenvalues counting multiplicity.
Therefore, the density condition $(\cL_n)\in\Lambda(\formT)$ implies that for all sufficiently large $n\in\mathbb{N}$ there are precisely
$\kappa+l_1(m) - 1$ elements from $\Spec(T,\cL_n)$  which are negative.
Since $\delta_{\formT}(\cE,\cL_n)=\varepsilon_n\to 0$, there are precisely $l_2(m)-l_1(m)+1$ ($=\dim\cE$) elements from
$\Spec(T,\cL_n)$ which are non-negative and of the order $\mathcal{O}(\varepsilon_n^2)$. The result now follows from \eqref{galerkin2} and \eqref{derivative}.
\end{proof}


\subsection{Lower bounds via Schur complement} \label{evlb}
In the previous section we found upper bounds for an eigenvalue $\lambda_m\in(d,\lambda_{\mathrm{e}})$. We now turn our attention to finding complementary lower bounds for this eigenvalue via the method described in
Section~\ref{zimmermethod}.

\begin{lemma}\label{1}
For $\mu\in(d,\lambda_{\mathrm{e}})$ let $0<\varepsilon<\mu-d$. If
$[-\varepsilon,0]\cap\Spec(S(\mu))\ne\varnothing$, then $[\mu-\varepsilon,\mu]\cap\Spec(M)\ne\varnothing$.
\end{lemma}
\begin{proof}
Let $\delta=\max\{\Spec(S(\mu))\cap[-\varepsilon,0]\}$ and assume that $[\mu-\varepsilon,\mu]\cap\Spec(M)=\varnothing$. According to \cite[Lemma 2.6]{math0},
 $\kappa(\mu) = \kappa(\mu-\varepsilon)$. By virtue of \eqref{derivative} and the Rayleigh-Ritz variational principle,
\[
0>\inf_{\efrac{V\subset\Dom(\frak{s})}{\dim V=\kappa(\mu)}}\sup_{\efrac{u\in V}{u\ne 0}} \frac{\frak{s}(\mu-\varepsilon)[u]}{\| u\|^2}
\ge\varepsilon + \inf_{\efrac{V\subset\Dom(\frak{s})}{\dim V=\kappa(\mu)}}\sup_{\efrac{u\in V}{u\ne 0}} \frac{\frak{s}(\mu)[u]}{\| u\|^2}
=\varepsilon + \delta,
\]
where the right hand side is non-negative. The result follows from the contradiction.
\end{proof}

By applying Lemma~\ref{upbdslem}, we can find $\mu\in\mathbb{R}$ such that $\lambda_m\le\mu$. If $S(\mu)$ has $\kappa + l_2(m)$ non-positive eigenvalues and $b>0$ is such that $(0,b]\cap \Spec(S(\mu))=\varnothing$, then we employ the Zimmermann-Mertins method with $T=S(\mu)$
to obtain a lower bound on the first non-positive eigenvalue. Combined with Lemma~\ref{1}, this yields a lower bound for
$\lambda_m$. We now find the rate of convergence of this lower bound.

\begin{lemma}\label{lemres1}
Let $b>0$ be such that $(0,b]\cap\Spec(S(\lambda_m))=\varnothing$. Let $\mu_n$ be a sequence of real numbers such that
$0\le\mu_n-\lambda_m=\varepsilon_n\to 0$ as $n\to \infty$. For all $n$ sufficiently large
$(0,b]\cap\Spec(S(\mu_n))=\varnothing$. Moreover,
\begin{equation} \label{reso_Schur}
\|(S(\lambda_m)-b)^{-1} - (S(\mu_n)-b)^{-1}\| = \mathcal{O}(\varepsilon_n).
\end{equation}
\end{lemma}
\begin{proof}
We first show that $b\not \in \Spec(S(\mu_n))$ for all sufficiently large $n\in\mathbb{N}$,
and that \eqref{reso_Schur} holds true. Let $x\in\Dom(\frak{s})$ and $\alpha_n=\mu_n-\lambda_m$, then
\begin{align*}
\frak{s}(\lambda_m)[x] &= \frak{a}[x] - \lambda_m\| x\|^2 - \langle(D-\lambda_m)^{-1}B^*x,B^*x\rangle\\
&=\frak{s}(\mu_n)[x] + \alpha_n\| x\|^2 + \langle[(D-\mu_n)^{-1} - (D-\lambda_m)^{-1}]B^*x,B^*x\rangle\\
&=\frak{s}(\mu_n)[x] + \alpha_n\| x\|^2+
 \alpha_n\langle(D-\mu_n)^{-1}(D-\lambda_m)^{-1}B^*x,B^*x\rangle  .
\end{align*}
Set $\hat{\frak{s}} = \frak{s}(\mu_n)-\frak{s}(\lambda_m)$. Note that
$\frak{a}[x]\le\frak{s}(\lambda)[x]+\lambda\|x\|^2$. By virtue of \eqref{domcons}, we have
\begin{align*}
|\hat{\frak{s}}[x]|&= \alpha_n\| x\|^2+
\alpha_n\langle(D-\mu_n)^{-1}(D-\lambda_m)^{-1}B^*x,B^*x\rangle\\
&\le \alpha_n\| x\|^2+
\alpha_n\|(D-\mu_n)^{-1}(D-\lambda_m)^{-1}\|\| B^*x\|^2\\
&\le \alpha_n\| x\|^2+
\alpha_n\|(D-\mu_n)^{-1}(D-\lambda_m)^{-1}\|(\alpha\frak{a}[x]+\beta\| x\|^2)\\
&\leq \alpha_n(a_n\| x\|^2+b_n\frak{s}(\lambda_m)[x]).
\end{align*}
where
\begin{align*}
a_n&= 1+\beta\|(D-\mu_n)^{-1}(D-\lambda_m)^{-1}\| + \alpha\lambda_m \to 1+\beta(\lambda_m-d)^{-2}+\alpha \lambda_m
\\
b_n&=\alpha\| (D-\mu_n)^{-1}(D-\lambda_m)^{-1}\|\to \alpha(\lambda_m-d)^{-2}\quad\textrm{as}\quad n\to\infty.
\end{align*}
Set  $c_1 = \max\{\|(S(\lambda_m) - b)^{-1}\|,\| S(\lambda_m)(S(\lambda_m) - b)^{-1}\|\}$, and let $n\in\mathbb{N}$ be
sufficiently large to ensure that $\alpha_n(a_n+b_n)<c_1^{-1}$. By virtue of \cite[Theorem VI-3.9]{katopert}, we obtain $b\not\in\Spec(S(\mu_n))$ and
\begin{equation}\label{normres}
\|(S(\lambda_m)-b)^{-1} - (S(\mu_n)-b)^{-1}\|\le
\frac{4\alpha_n(a_n+b_n)c_1^2}
{(1-\alpha_n(a_n+b_n)c_1)^2}
\end{equation}
which immediately implies \eqref{reso_Schur}.

It remains to show that $(0,b]\cap\Spec(S(\mu_n))=\varnothing$ for all sufficiently large $n\in\mathbb{N}$.
By virtue of \cite[Theorem VII-4.2]{katopert}, there exists a constant
$\nu<\min\Spec(S(\mu_n))$ for all sufficiently large $n\in\mathbb{N}$.
Let $\Gamma$ be a circle with center $(b + \nu)/2$ and
radius $(b-\nu)/2$, and set
\[
c_2 = \max_{z\in\Gamma}\Big\{\|(S(\lambda_m) - z)^{-1}\|,\| S(\lambda_m)(S(\lambda_m) - z)^{-1}\|\Big\}.
\]
Then $\alpha_n(a_n+b_n)<c_2^{-1}$ for all sufficiently large $n\in\mathbb{N}$. Applying the same argument as above, we obtain
\begin{equation*}
\max_{z\in\Gamma}\Big\{\|(S(\lambda_m)-z)^{-1} - (S(\mu_n)-z)^{-1}\|\Big\}\le
\frac{4\alpha_n(a_n+b_n)c_2^2}
{(1-\alpha_n(a_n+b_n)c_2)^2}.
\end{equation*}
The right hand side of converges to zero as $n\to\infty$. Thus,
the spectral subspaces of $S(\lambda_m)$ and $S(\mu_n)$ corresponding to the eigenvalues below $b$ have the same dimension
for all sufficently large $n\in\mathbb{N}$.
The desired conclusion follows from \eqref{derivative} and the
Rayleigh-Ritz variational principle.
\end{proof}

According to Theorem \ref{super}, if
$(\cL_n)\in\Lambda(\frak{s}(\lambda_m))$ and $\delta_{\frak{s}(\lambda_m)}(\cE,\cL_n)=\mathcal{O}({\varepsilon}_n)$,
then we obtain a sequence of upper bounds
\begin{equation} \label{upper_seq}
\mu_n^+\searrow\lambda_m \qquad \text{ satisfying } \qquad
\mu_n^+-\lambda_m = \mathcal{O}({\varepsilon}_n^2).
\end{equation}
 If we now pick $b$ satisfying the hypothesis of Lemma~\ref{lemres1},
by virtue of \eqref{zimbound} we find a lower bound for the smallest in modulus non-positive eigenvalue of $T=S(\mu_n^+)$. That is
\[
b + \frac{1}{\tau_n^-} \le \max \left\{\Spec(S(\mu_n^+))\cap(-\infty,0]\right\}.
\]
Lemma~\ref{1} ensures corresponding lower bounds
\begin{equation} \label{lower_seq}
\mu_n^- = \mu_n^+ + b+1/\tau_n^-\le\lambda_m.
\end{equation}
In the theorem below we find bounds on the speed of convergence $\mu_n^-\to \lambda_m$.

Before proceeding further, we note that $S(\mu)$ and $B(D-\lambda_m)^{-1}(D-\mu)^{-1}B^*$ for the plane slab configuration 
\eqref{slab} are sectorial Sturm-Liouville operators for all $\mu\in U$. Both families of operators are closed in the domain
\begin{equation} \label{domain_schur}
\mathcal{D}=H^{2}((0,1);\ud x)\cap H_{0}^{1}((0,1);\ud x),
\end{equation}
which coincides with \eqref{schurdom} and is independent of $\mu$. Moreover, they are both holomorphic families of type (A), see \cite[Example~VII-2.12]{katopert}.   

\begin{theorem}\label{enclose}
Suppose that the entries of $M_0$ satisfy \ref{top_1} - \ref{top_3}.
Assume that $S(\mu)$ is a holomorphic family of type (A) with $\Dom(S(\mu))=\mathcal{D}$ independent of $\mu$. Assume additionally that $B(D-\lambda_m)^{-1}(D-\mu)^{-1}B^*$ is closed on $\mathcal{D}$. Let $(\cL_n)\in\Lambda(S(\lambda_m))$ with 
$\delta_{S(\lambda_m)}(\cE,\cL_n)=\varepsilon_n$.
If $\mu_n^-$ is constructed as in \eqref{lower_seq}, then $\mu_n^-\leq \lambda_m$ and
$\lambda_m -\mu_n^- = \mathcal{O}(\varepsilon_n^2)$ as $n\to \infty$.
\end{theorem}
\begin{proof}
Let $\mu\in(d,\lambda_{\mathrm{e}})$, $x\in\mathcal{D}$ and $\nu<\min\Spec(A)$. According to \eqref{schuract} we have
\begin{align*}
S(\mu)x - S(\lambda_m)x &= (A-\nu)(x - \overline{(A-\nu)^{-1}B}(D-\mu)^{-1}B^*x) + (\nu-\mu)x\\
&\quad - (A-\nu)(x - \overline{(A-\nu)^{-1}B}(D-\lambda_m)^{-1}B^*x) - (\nu-\lambda_m)x\\
&= (A-\nu)(\overline{(A-\nu)^{-1}B}[(D-\lambda_m)^{-1} - (D-\mu)^{-1}]B^*x)\\
&\quad + (\lambda_m-\mu)x\\
&= (\lambda_m-\mu)(A-\nu)(\overline{(A-\nu)^{-1}B}(D-\lambda_m)^{-1}(D-\mu)^{-1}B^*x)\\
&\quad + (\lambda_m-\mu)x\\
&= (\lambda_m-\mu)B(D-\lambda_m)^{-1}(D-\mu)^{-1}B^*x+ (\lambda_m-\mu)x.
\end{align*}
We consider the closed operator $B(D-\lambda_m)^{-1}(D-\mu)^{-1}B^*$. Since
\begin{align*}
B(D-\lambda_m)^{-1}(D-\mu)^{-1}B^*x = \frac{S(\mu)x - S(\lambda_m)x}{\lambda_m - \mu} - x\quad\textrm{for}\quad\mu\ne\lambda_m,
\end{align*}
and $S(\mu)$ is holomorphic of type (A), then
$\Vert B(D-\lambda_m)^{-1}(D-\mu)^{-1}Bx\Vert$ is uniformly bounded in a neighbourhood of $\mu$. 
Moreover, for any $y\in\Dom(B^*)$ the function
\begin{displaymath}
\langle B(D-\lambda_m)^{-1}(D-\mu)^{-1}B^*x,y\rangle = 
\langle (D-\mu)^{-1}B^*x,(D-\lambda_m)^{-1}B^*y\rangle
\end{displaymath}
is analytic for $\mu\in U$. It follows that 
$B(D-\lambda_m)^{-1}(D-\mu)^{-1}B^*$ is a holomorphic family of type (A). 

Let $J\subset(d,\lambda_{\mathrm{e}})$ be any compact interval containing a neighbourhood of $\lambda_m$. 
By virtue of \cite[Section VII.2.1]{katopert}, there always exists a constant $c_3>1$ such that
\begin{equation}\label{M}
\frac{\|x\| + \| B(D-\lambda_m)^{-1}(D-\mu)^{-1}B^*x\|}{\|x\| + \| B(D-\lambda_m)^{-2}B^*x\|} \le c_3\quad\textrm{for all}\quad \mu\in J.
\end{equation}
Since the operators $B(D-\lambda_m)^{-2}B^*$ and $S(\lambda_m)$ have the same domain $\mathcal{D}$, there exist constants $\tilde{\alpha},\tilde{\beta}\ge 0$ such that
\begin{align}
\| B(D-\lambda_m)^{-2}B^*x\|&\le \tilde{\alpha}\| S(\lambda_m)x\| +
\tilde{\beta}\|x\|\quad\textrm{for all}\quad x\in\mathcal{D}\label{consts}.
\end{align}
As $\mu_n^+-\lambda=\mathcal{O}(\varepsilon_n^2)$, for some $N\in\mathbb{N}$ large enough $\mu_n^+\in J$ whenever $n\ge N$. Combining \eqref{consts} with \eqref{M}, gives
\begin{align*}
\| B(D-\lambda_m)^{-1}(D-\mu_n^+)^{-1}B^*x\|&\le (c_3-1)\|x\| + c_3\| B(D-\lambda_m)^{-2}B^*x\|\\
&\le (c_3-1)\|x\| + \tilde{\alpha}c_3\|S(\lambda_m)x\| + \tilde{\beta}c_3\|x\|\\
&\le c_4(\| S(\lambda_m)x\| + \|x\|),
\end{align*}
where $c_4\ge 0$ is independent of $n\ge N$. Thus
\begin{equation}\label{mess1}
\| S(\mu_n^+)x - S(\lambda_m)x\| \le (\mu_n^+-\lambda_m)(c_4+1)(\| S(\lambda_m)x\|+\|x\|)\quad\textrm{for all}\quad n\ge N.
\end{equation}

Let $\hat{\cL}_n = (S(\mu_n^+)-b)\cL_n$. We show that
$(\hat{\cL}_n)\in\Lambda$. Let $v\in\mathcal{H}$. There exists $u\in\mathcal{D}$ such that
$(S(\lambda_m)-b)u = v$. Since $(\cL_n)\in\Lambda(S(\lambda_m))$ we have a sequence $u_n\in\cL_n$
satisfying $(S(\lambda_m)-b)u_n \to v$. As $b\not \in \Spec(S(\lambda_m))$, the 
sequences $\| u_n\|$ and $\| S(\lambda_m)u_n\|$ are uniformly bounded. Hence, it follows from \eqref{mess1} that 
$\|S(\mu_n^+)u_n-S(\lambda_m)u_n\|\to 0$. Since
\[
v - (S(\mu_n^+)-b)u_n +  (S(\mu_n^+)-S(\lambda_m))u_n = v - (S(\lambda_m)-b)u_n \to 0,
\]
clearly also $(S(\mu_n^+)-b)u_n\to v$. Thus $\hat{\cL}_n\in\Lambda$.

We now show that $\delta(\cE,\hat{\cL}_n)=\mathcal{O}(\varepsilon_n)$. Let $\phi_1,\dots,\phi_{k}$ be an orthonormal basis for $\cE$. There exist vectors $u_{n,j}\in\cL_n$, such that
$\|S(\lambda_m)(\phi_j - u_{n,j})\|\le \varepsilon_n$ and $\|\phi_j - u_{n,j}\|\le\varepsilon_n$ for each $1\le j\le k$. We set $\hat{u}_{n,j} = (S(\mu_n^+)-b)u_{n,j}\in\hat{\cL}_n$. Using \eqref{mess1} we have for any normalised $\phi\in\cE$
\begin{align*}
\Big\|\phi + \sum_{j=1}^k\frac{\langle\phi_j,\phi\rangle}{b}\hat{u}_{n,j}\Big\| &=
\Big\|\sum_{j=1}^k\langle\phi_j,\phi\rangle\phi_j +
\sum_{j=1}^k\frac{\langle\phi_j,\phi\rangle}{b} (S(\mu_n^+)-b)u_{n,j}\Big\|\\
&\le\Big\|\sum_{j=1}^k\langle\phi_j,\phi\rangle(\phi_j -u_{n,j})\Big\| + 
\Big\| \sum_{j=1}^k\frac{\langle\phi_j,\phi\rangle}{b}S(\mu_n^+)u_{n,j}\Big\|\\
&\le k\varepsilon_n + \frac{1}{b}\sum_{j=1}^k\left(\|S(\mu_n^+)u_{n,j} - S(\lambda_m)u_{n,j}\| + \|S(\lambda_m)u_{n,j}\|\right)\\
&\le 2k\varepsilon_n + \frac{1}{b}\sum_{j=1}^k (\mu_n^+-\lambda_m)(c_4+1)(\| S(\lambda_m)u_{n,j}\|+\|u_{n,j}\|)\\
&\le 2k\varepsilon_n + \frac{1}{b}\sum_{j=1}^k (\mu_n^+-\lambda_m)(c_4+1)(\varepsilon_n+\|u_{n,j}\|).
\end{align*}
Therefore $\delta(\cE,\hat{\cL}_n)=\mathcal{O}(\varepsilon_n)$.

We complete the proof of the theorem as follows. By applying \eqref{galerkin2} to the operator $T=(S(\lambda_m)-b)^{-1}$ and
eigenvalue $(-b)^{-1}=\min\{\Spec((S(\lambda_m) - b)^{-1})\}$, we obtain
\begin{equation}\label{gal1}
\min\{\Spec((S(\lambda_m)-b)^{-1},\hat{\cL}_n)\} + b^{-1} = \mathcal{O}(\varepsilon_n^2).
\end{equation}
Using Lemma~\ref{lemres1} and $0\le\mu_n^+-\lambda_m=\mathcal{O}(\varepsilon_n^2)$, we have 
\begin{equation}\label{gal2}
\|(S(\lambda_m)-b)^{-1} - (S(\mu_n^+)-b)^{-1}\|=\mathcal{O}(\varepsilon_n^2).
\end{equation}
From \eqref{gal1}, \eqref{gal2} and the Rayleigh-Ritz variational principle, it becomes clear that
\begin{equation}\label{gal3}
\tau_n^- = \min\{\Spec((S(\mu_n^+)-b)^{-1},\hat{\cL}_n)\}\quad\textrm{satisfies}\quad\tau_n^- + b^{-1} = 
\mathcal{O}(\varepsilon_n^2).
\end{equation}
Moreover, $b + 1/\tau_n^-$ is precisely the lower bound on the smallest in modulus non-positive eigenvalue 
of $S(\mu_n^+)$ which is obtained from the
Zimmermann-Mertins method. Then $\mu_n^- = \mu_n^+ + b+1/\tau_n^-\le\lambda_m$ follows from Lemma \ref{1}, and $\lambda_m -\mu_n^- = \mathcal{O}(\varepsilon_n^2)$ follows from \eqref{gal3}.
\end{proof}


\begin{figure}[tt]
\centerline{\includegraphics[height=8cm, angle=0]{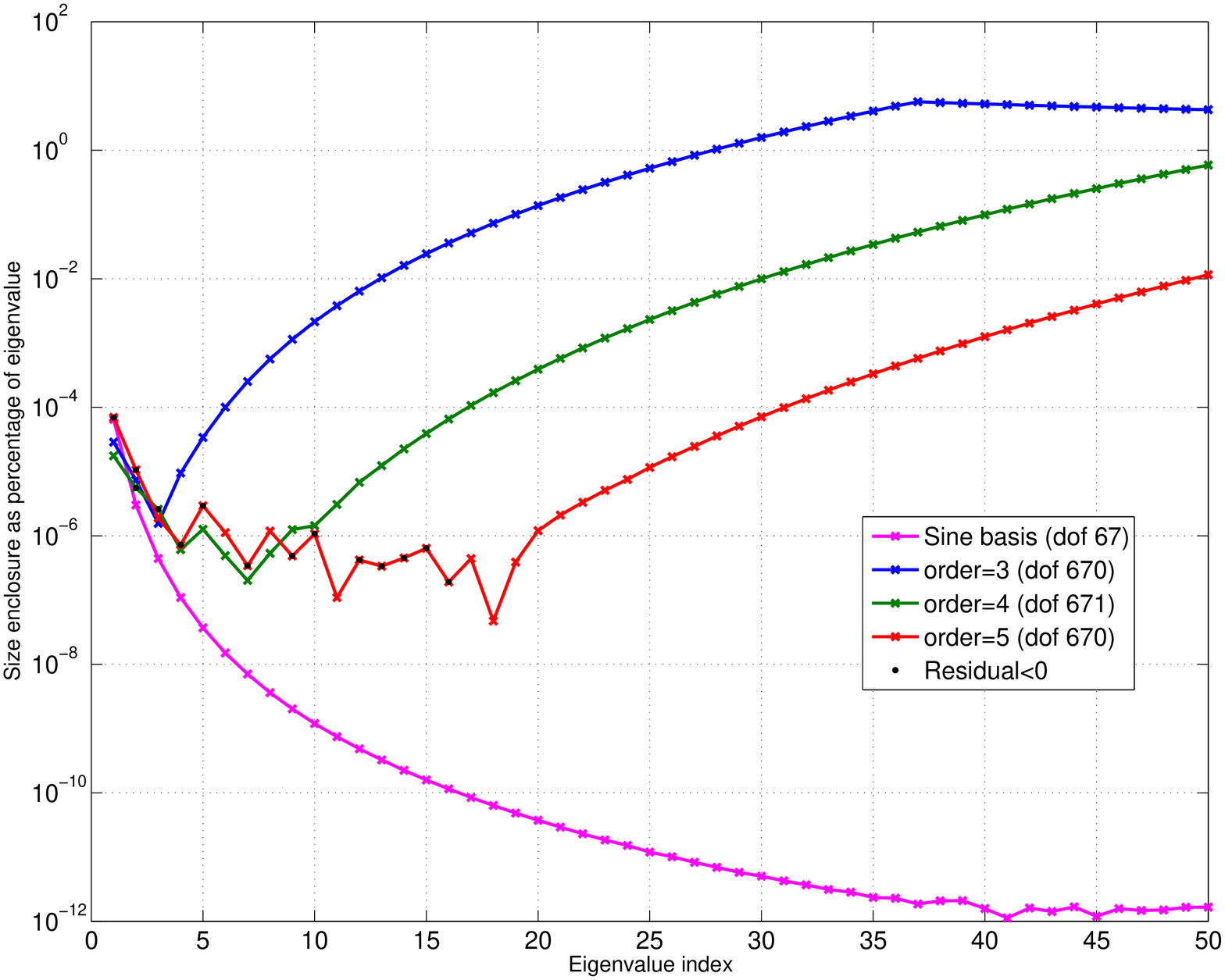}}
\caption{Log-log graph. Vertical axis:  $\frac{\mu_n^+-\mu_n^-}{\mu_n^-}\times 100$. Horizontal axis: eigenvalue index $m$. We depict the relative size of the enclosure in the calculation of the first 50 eigenvalues of $M_{\slb}$ of Example~\ref{non-constant_coeff}. The subspace $\cL_n$ is chosen to be:
$\cL_{67}^{\sin}$ and $\cL(h,1,r)$ for Hermite elements of order
$r=3$, 4 and 5 on an uniform mesh with $h$ chosen so the dimension of the spaces is
approximately $10\times 67$. We have chosen $\tau_{\mathrm{i}}=10^{-14}$, $\tau_{\mathrm{s}}=10^{-12}$, $\tau_{\mathrm{b}}=\mathcal{O}(10^{-5})$ in the case of the sine basis, and $\tau_{\mathrm{i}}=10^{-10}$, $\tau_{\mathrm{s}}=10^{-6}$ and $\tau_{\mathrm{b}}=\mathcal{O}(10^{-3})$ in the case of the finite element method.
 \label{fig2}}
\end{figure}

\begin{figure}[t]
\centerline{\includegraphics[height=8cm, angle=0]{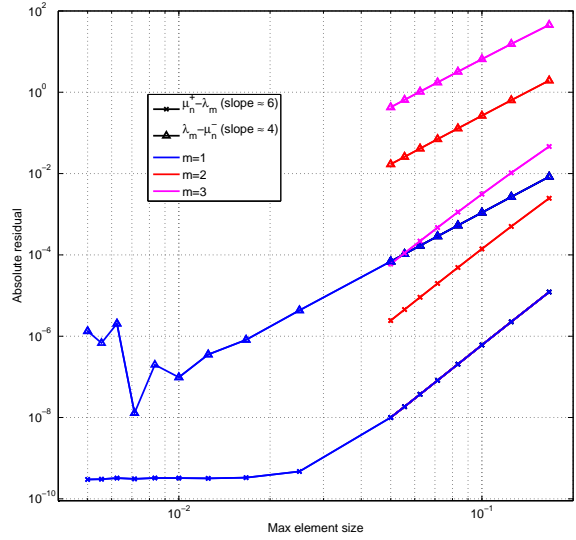}}
\caption{Log-log graph. Vertical axis: $|\lambda^{\exctt}_{1,2,3}-\mu^{\pm}_{n,h}|$. Horizontal axis: maximum element size $h$. The subspace $\cL_n$ is chosen to be:
$\cL(h,1,3)$ for decreasing values of $h$. For these calculations $\tau_{\mathrm{i}}=10^{-10}$, $\tau_{\mathrm{s}}=10^{-12}$, $\tau_{\mathrm{b}}=\mathcal{O}(10^{-3})$.\label{fig4}}
\end{figure}

\section{Numerical examples: plane slab configuration}    \label{numerical_slab}

An optimal strategy in terms of convergence for calculating enclosures for the configuration \eqref{slab} can be established from the approach in Section~\ref{supersection}. We now illustrate the practical applicability of this strategy by performing various numerical experiments on benchmark models. Our equilibrium quantities will be chosen from examples~\ref{constant_coeff} and \ref{non-constant_coeff}.

For a fixed $\mu\in(d,\infty)$, the eigenvalues of $S(\mu)$ are simple in both examples. The corresponding eigenvectors are in $C^{\infty}(0,1)$ and they satisfy Dirichlet boundary conditions at the endpoints of the interval. We have
\[
   \Spec(M_{\slb})\cap(d,\infty)=\{\lam_1<\lam_2<\ldots\}
\]
where each eigenvalue is simple and $\lam_j\to\infty$.
Below we distinguish the model used by denoting these eigenvalues by $\lam^{\exctt}_m$ and $\lam^{\exvar}_m$ respectively. For all $u\in\Dom(\frak{s})=H^1_0((0,1);\ud x)$,
\begin{equation} \label{Schur_comp_pd}
\begin{aligned}
\frak{s}(\mu)[u] & \ge \frak{a}[u] - \mu\| u\|^2 \\
& \ge \pi^2\| u\|^2 + \langle(7/4 - x)u,u\rangle-\mu\| u\|^2 \\
& \ge (\pi^2 -\mu)\| u\|^2.
\end{aligned}
\end{equation}
Thus $S(\mu)$ is positive definite for $\mu\in(d,\pi^2)$. Hence upper bounds $\mu_n^+$ for $\lambda_m$ can be found from Theorem \ref{super} with $\kappa = 0$.

In practice we find $\mu_n^+$ as follows.
For a fixed $\cL_n$ we compute a few eigenvalues of $S(\mu)$ for $\mu$ in an uniform partition with $p$ points of a suitable interval $(a,b)$ containing only $\lambda_m$. We then approximate $\mu_n^+$ via one iteration of Newton's method. Below, the integrations involved in the assembling
of the matrix problems are set to a tolerance of the order $\mathcal{O}(\tau_{\mathrm{i}})$,
the eigenvalue solver is set to a tolerance of order $\mathcal{O}(\tau_{\mathrm{s}})$
and $\tau_{\mathrm{b}}=(b-a)/p$. These are  different for the different experiments.
By virtue of \eqref{derivative}, the root finding step is accurate to $\mathcal{O}(\tau_{\mathrm{b}}^2)$.

To find complementary lower bounds from
Theorem \ref{enclose}, we require $b>0$ such that $(0,b]\cap\Spec(S(\mu_n^+))=\varnothing$
for all sufficiently large $n\in\mathbb{N}$. From  \eqref{Schur_comp_pd}
 it follows that $\frak{s}(\mu_n^+)[u] \ge \frak{a}[u] - \mu_n^+\| u\|^2$ for all
$u\in\Dom(\frak{s})$ (in both Example~\ref{constant_coeff} and Example~\ref{non-constant_coeff}). By the minmax principle, the $(m+1)$-th eigenvalue of $S(\mu_n^+)+\mu_n^+$ lies above the $(m+1)$-th eigenvalue of $A$. In fact
$\lambda_m<(m+1)^2\pi^2<\lambda_{m+1}$ and we may choose $b\in(0,(m+1)^2\pi^2 - \mu_n^+)$.
Integration and eigenvalue solver tolerances are set as for the upper bounds.

We consider two canonical basis to generate
$\cL_n\subset \mathcal{D}$, see \eqref{domain_schur}. A first natural choice is the sine basis,
\[
  \cL_n^{\sin}=\Span\{u_1,\dots,u_n\}\qquad
   \text{where} \qquad u_n = \sqrt{2}\sin(n\pi x).
\]
Standard arguments show that
$\cL_n\in\Lambda(S(\lambda_m))$ and
\[
\delta_{S(\lambda_m)}(\ker S(\lambda_m),\cL_n)=\mathcal{O}(n^{-r})
\]
where $r$ can be chosen arbitrarilly large. Applying Theorem \ref{super} and Theorem \ref{enclose} we obtain
\begin{displaymath}
\mu_n^+\searrow\lambda_m,\quad\mu_n^-\nearrow\lambda_m,\quad\textrm{and}\quad
\mu_n^+-\mu_n^-=\mathcal{O}(n^{-r}).
\end{displaymath}
This means that the enclosures should converge to zero super-polynomially fast
for the family of subspaces $\cL_n^{\sin}$. See Table~\ref{table1} and Figure~\ref{fig2}.
All calculations involving this basis were coded in Matlab.

\begin{table}[hhh]
\begin{tabular}{c|c|c|c|c|c}
n & $\lambda^{\exvar}_1$ & $\lambda^{\exvar}_2$ & $\lambda^{\exvar}_3$ & $\lambda^{\exvar}_4$ & $\lambda^{\exvar}_5$  \\
\hline
& & & & &\\
5 & $12.350^{47799}_{38099}$ & $41.9106^{5300}_{3750}$ & $91.2474^{7057}_{6613}$ &  $160.3305^{8480}_{7817}$ &  $ 249.155^{19069}_{07400}$\\
& & & & &\\
10 & $12.3504^{7592}_{2524}$ & $41.9106^{5224}_{4418}$ &  $91.2474^{7031}_{6778}$ &  $160.33058^{264}_{158}$ &  $249.155079^{76}_{13}$ \\
& & & & &\\
20 &  $12.3504^{7563}_{4946} $& $41.9106^{5214}_{4796}$ &  $91.2474^{7026}_{6895}$ &   $160.330582^{62}_{04}$ &  $249.155079^{73}_{43}$ \\
& & & & &\\
40 & $12.3504^{7559}_{6228}$  &  $41.91065^{213}_{001}$ & $91.2474^{7026}_{6958}$ & $160.330582^{61}_{31}$ &  $249.155079^{73}_{57}$
\end{tabular}
\vspace{5mm}
\caption{Approximation of the first five eigenvalues of $M_{\slb}$ for
Example~\ref{non-constant_coeff} and test spaces chosen as $\cL_n^{\mathrm{s}}$.
For these calculations $\tau_{\mathrm{i}}=10^{-14}$, $\tau_{\mathrm{s}}=10^{-12}$, $\tau_{\mathrm{b}}=\mathcal{O}(10^{-5})$.
\label{table1}}
\end{table}

\begin{figure}[t]
\centerline{\includegraphics[height=8cm, angle=0]{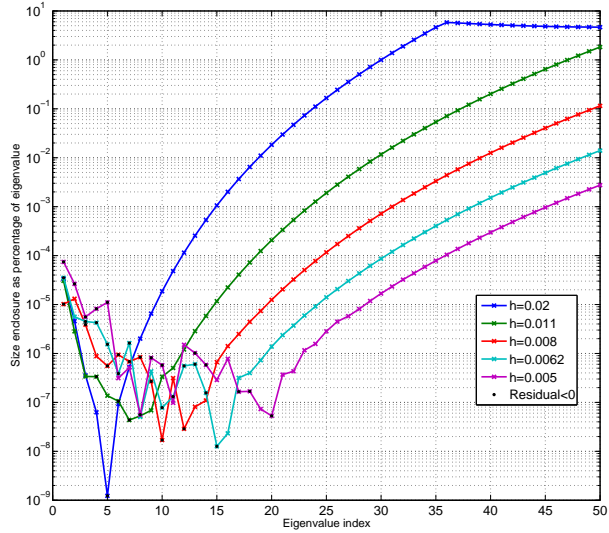}}
\caption{Log-log graph. Vertical axis:  $\frac{\mu_n^+-\mu_n^-}{\mu_n^-}\times 100$. Horizontal axis: eigenvalue index $m$. We depict the relative size of the enclosure in the calculation of the first 50 eigenvalues of $M_{\slb}$ of Example~\ref{non-constant_coeff}. The subspace $\cL_n$ is chosen to be:
$\cL(h,1,5)$ for decreasing values of $h$. For these calculations $\tau_{\mathrm{i}}=10^{-10}$, $\tau_{\mathrm{s}}=10^{-6}$, $\tau_{\mathrm{b}}=\mathcal{O}(10^{-3})$. \label{fig3}}
\end{figure}

Another natural basis is obtained by applying the finite element method.
Let  $\Xi$ be an equidistant partition of $[0,1]$ into $n$ sub-intervals $I_l=[x_{l-1},x_l]$ of length $h=1/n=x_l-x_{l-1}$. Consider the subspaces
\begin{equation}   \label{fe_subspaces}
    \cL(h,k,r)=V_h(k,r,\Xi)=\{v\in C^k(0,1): v\upharpoonright _{I_l}\in P_r(I_l),
    1\leq l\leq n, v(0)=0=v(1) \}.
\end{equation}
The $\cL(h,k,r)$ are the finite element spaces generated by $C^k$-conforming elements of  order $r$ subject to Dirichlet boundary conditions at $0$ and $1$.  Then
\[
    \|v-v_h\|_{W^{p,2}(0,1)}\leq c\|v\|_{H^{p+1}_\Xi(0,1)} h^{r+1-p}
\]
for $v_h\in \cL(h,k,r)$ the finite element interpolant of $v\in C^k\cap H^{r+1}_\Xi(0,1)$.
For fixed $k,r\ge 1$, let $\mu^+_{m,h}$ and $\mu^-_{m,h}$ be the upper and lower bounds for $\lambda_m$ given by Theorem \ref{super}  and Theorem \ref{enclose}, respectively. Then
\begin{equation}\label{order_upper_bound}
\mu^+_{m,h} -\lambda_m = \mathcal{O}(h^{2r}) \quad\textrm{and}\quad\mu^-_{m,h}-\lambda_m=\mathcal{O} (h^{2(r-1)}).
\end{equation}
All calculations involving this basis were coded in Comsol.

Figure~\ref{fig4} shows that the orders of convergence found in \eqref{order_upper_bound} are optimal in the case $r=3$ for Hermite elements and $m=1,2,3$. In order to compare the quality of the upper and lower bounds, we have chosen Example~\ref{constant_coeff}
and calculated the value of $\lambda^{\exctt}_m$ with the exact formula in machine precision. Observe that the upper bounds are all roughly 4 orders of magnitude more accurate than the lower bounds. This is certainly expected from the fact that
the calculation of the upper bound involves the solution of a second order
problem, whereas that of the lower bound involves the  eigenproblem \eqref{zimmer}
with $T=S(\mu^+_n)$ which is of fourth order. Here we have purposely chosen large values of $h$, so the calculation of the bounds for $\lambda_2$ and $\lambda_3$ is not particularly accurate.

The aim of the experiment performed in Figure~\ref{fig2} is to compare accuracies
in the computation of the bounds by picking $\cL_n$ of roughly the same dimension, but generated
by different bases. For this we have fixed $\cL_n$ of a given dimension and compute the size of the enclosure $(\mu^-_n,\mu^+_n)$ relative to the size of the lower bound $\mu^-_n$. We consider Example~\ref{non-constant_coeff}. We have chosen $\dim \cL_n=67$ for the sine basis and $\dim \cL_n\approx 670$ for the finite element bases (remember that the sine basis is exponentially accurate).

The accuracy deteriorates (even in relative terms) as the eigenvalue counting number $m$ increases. For the same dimension of $\cL_n$, accuracy increases as the order of the polynomial $r$ increases. In this figure, the enclosures found for $\lambda_m$ for $m<5$
($r=3$), $m<10$ ($r=4$) and $m<20$ ($r=5$) should not be trusted and it is just included for illustration purposes. This locking effect is consistent with the fact that the calculation of the enclosures can never be more accurate than a factor of
$\max\{\tau_{\mathrm{i}},\tau_{\mathrm{s}},\tau_{\mathrm{b}}^2\}$. 

We can examine this phenomenon in more detail from Figure~\ref{fig3} and the blue line in Figure~\ref{fig4}. As the dimension of the test subspace decreases, for each individual eigenvalue, the residual starts decreasing and eventually hits the accuracy threshold.
From Figure~\ref{fig4} it should be noted that the lower bound hits the threshold earlier than the upper bound, however this threshold for the lower bound is three to four orders of magnitude
larger that that of the lower bound.


\section{Numerical examples: cylindrical pinch configuration}

The approach considered in Section~\ref{supersection} cannot be implemented on the cylindrical pinch configuration for $m\not=0$ as the block operator matrix does not satisfy condition \ref{top_3}. We now report on a set of numerical experiments performed on the benchmark model in Example~\ref{example_cylinder}, by directly applying the method described in Section~\ref{zimmerman} to $T=M$.  

In this case we have chosen $\cL_n=\cL(h,1,r)\times \cL(h,1,r)$ where $\cL(h,1,r)$ is defined by \eqref{fe_subspaces} and is generated by Hermite elements. The Dirichlet boundary condition inposed at both ends of the interval $[0,1]$ ensures that $\cL_n \in \Dom(M)$.  In Table~\ref{tab1} we show computation of the first three eigenvalues 
above $\Spec_\ess(M)$. Similar calculations can be found in \cite[Table~1]{1991Kakoetal}. Note that in the latter, for $N=32$ the approximated eigenvalue appears to be below $\lambda_1$ whereas for $N=64$ it appears to be above $\lambda_1$. This phenomenon is not present in the method described in Section~\ref{zimmerman} as it always provide a certified enclosure for the eigenvalue. 

\begin{table}[hhh]
\begin{tabular}{c|c|cc|c}
$j$ & exact $\lambda_m$ & $(a,b)$ & enclosure & d.o.f. \\ \hline
$1$ & $4.38995771667$ & $(3,20)$ &  $4.3_{895445}^{903962}$  & $5004$    \\
$2$ & $29.4242820473$ &  $(20,60)$  & $29.42_{3873}^{4656}$ & $5720$   \\
$3$ & $73.8686971063$ &  $(60,100)$  & $73.86_{8030}^{9378}$ & $8004$
\end{tabular}
\caption{Enclosures for the first three eigenvalues in  $\Spec_\dsc (M)$ above the essential spectrum
for Example~\eqref{example_cylinder} by direct application of the method of Section~\ref{zimmerman}. For these calculations we have chosen $r=3$ and $\tau_{\mathrm{i}}=\tau_{\mathrm{s}}=10^{-6}$. \label{tab1} }
\end{table}

The eigenfunctions of $M$ associated to $\lambda_m$ possess a singularity at the origin, so neither the upper nor the lower bounds obey an estimate analogous to that of \eqref{order_upper_bound}. On the left of Figure~\ref{fig_cyl_1} we show a log-log plot of the size of the enclosure against maximum element size for $r=3$ and $r=5$. The graph clearly indicates that the order of decrease of the enclosure does not seem to decrease with the order of the polynomial.  On the right of Figure~\ref{fig_cyl_1} we show the absolute residuals for lower and upper bounds separately. Both graphs indicate that
\[
      |\lambda_m-\mu_{m,h}^\pm|=O(h^{1}) \qquad \text{as }h\to 0
\]
equally for $r=3$ and $r=5$.

\begin{figure}[t]
\centerline{\includegraphics[height=6cm, angle=0]{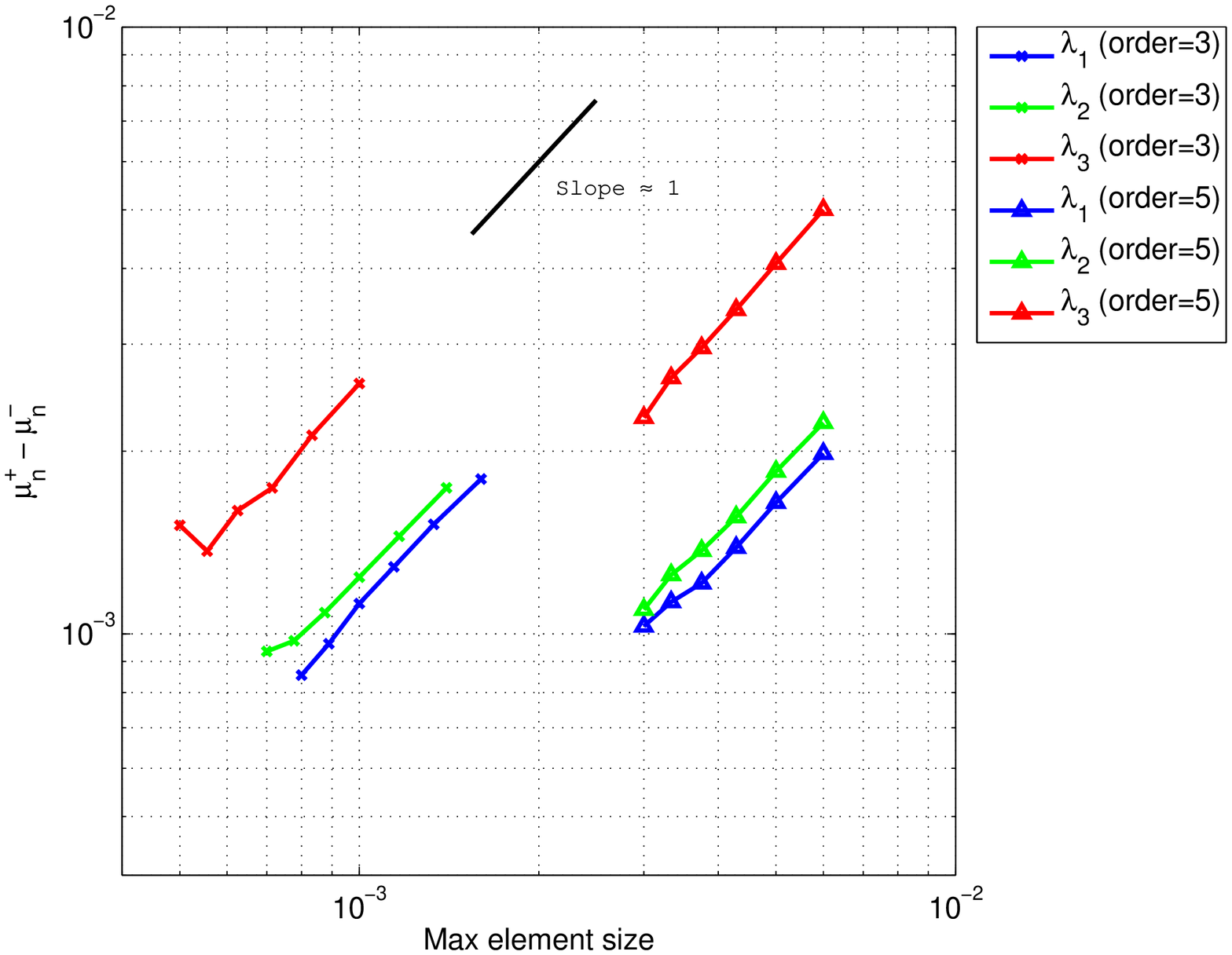} \hspace{-1cm}
\includegraphics[height=6cm, angle=0]{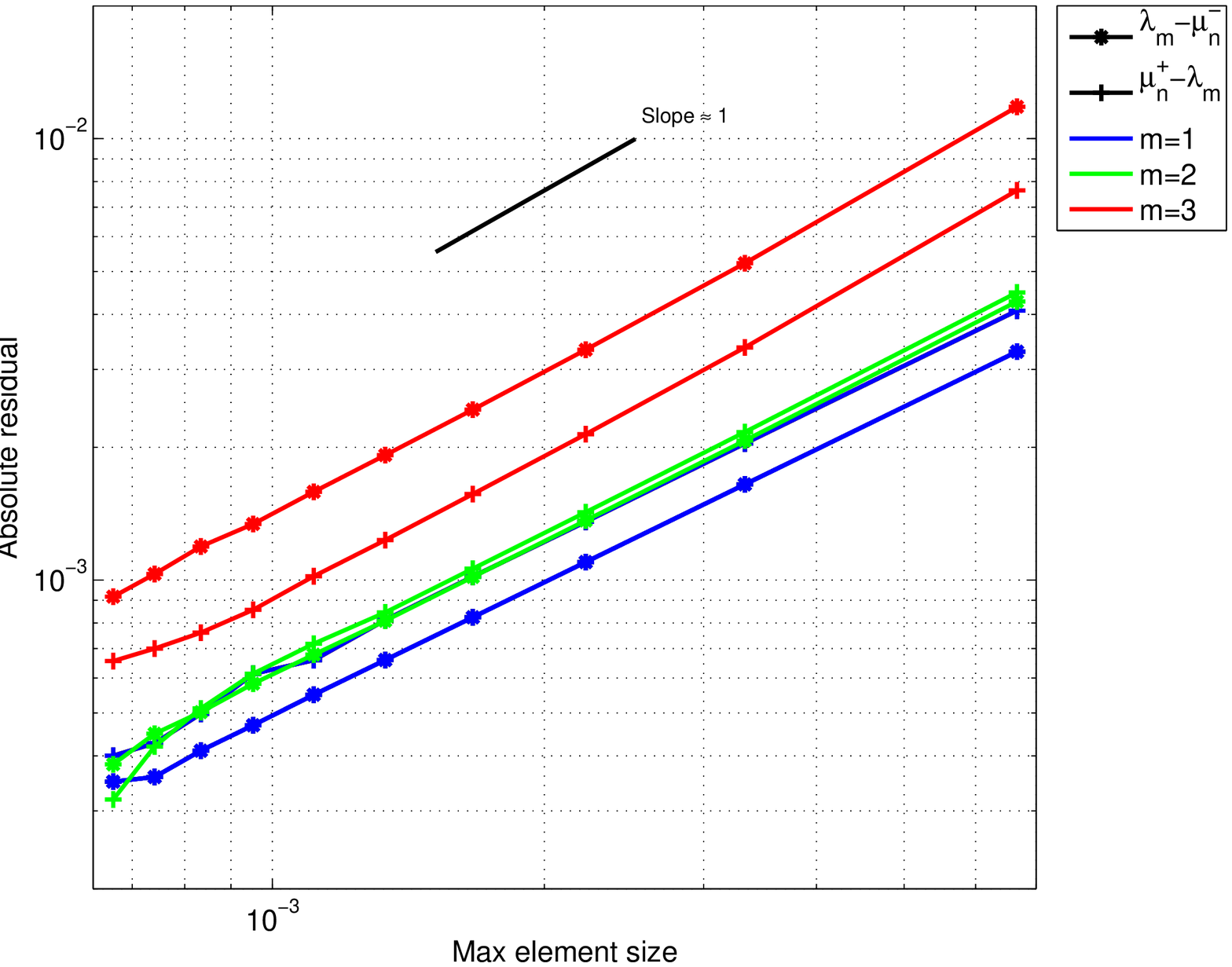}}
\caption{Log-log graphs. Vertical axis:  $\mu_{m,h}^+-\mu_{m,h}^-$ (left) and $|\lambda_m-\mu^\pm_{m,h}|$ (right). Horizontal axis: Maximum element size $h$.  For the subspaces 
$\cL_n$ we choose $r=3,5$ (left) and $r=3$ (right). Note that the order of decrease of all the residuals is roughly $O(h^1)$ for both polynomial orders (left) and both bounds (right). 
For these calculations $\tau_{\mathrm{i}}=\tau_{\mathrm{s}}=10^{-10}$. \label{fig_cyl_1}}
\end{figure}

\section*{Acknowledgements}
This research was funded by EPSRC grant number 113242.

\end{document}